\documentclass[letterpaper,10pt,journal]{IEEEtran}

\IEEEoverridecommandlockouts

\usepackage{amssymb,euscript,latexsym,amsmath,bm}
\usepackage{graphicx}
\usepackage[percent]{overpic}
\usepackage{rotating}

\usepackage{algorithm, algorithmic}
\usepackage{hyperref}

\renewcommand{\epsilon}{\varepsilon}

\makeatletter
\def\thefigure{%\thesection.
\arabic{figure}}
\def\fps@figure{h, t}
\def\theequation{%\thesection.
\arabic{equation}}
\makeatother
\DeclareMathOperator*{\ext}{ext}

\newtheorem{theorem}{Theorem}%[section]

\def\pder#1#2{{\frac {\partial #1} {\partial #2}}}

\begin{document}
\title{Controlled Lagrangians and Stabilization \\ of Discrete Mechanical Systems I}
\author{Anthony M. Bloch, Melvin Leok, Jerrold E. Marsden, and Dmitry V. Zenkov}
\maketitle

\begin{abstract}
\noindent
Controlled Lagrangian and matching techniques are developed for the stabilization of relative equilibria and equilibria of discrete mechanical systems with symmetry as well as broken symmetry.   Interesting new phenomena arise in the controlled Lagrangian approach in the discrete context that are not present in the continuous theory. In particular, to make the discrete theory effective, one can make an appropriate selection of momentum levels or, alternatively, introduce a new parameter into the controlled Lagrangian to complete the kinetic matching procedure. Specifically, new terms in the controlled shape equation that are necessary for potential matching in the discrete setting are introduced. The theory is illustrated with the problem of stabilization of the cart-pendulum system on an incline. The paper also discusses digital and model predictive controlers.
\end{abstract}

\section{Introduction}
The method of controlled Lagrangians for stabilization of relative equilibria (steady state motions) originated in Bloch, Leonard, and Marsden \cite{BLM1} and was then developed in Auckly~\cite{A}, Bloch, Leonard, and Marsden \cite{BLM2, BLM4, BLM3}, Bloch, Chang, Leonard, and Marsden \cite{BCLM}, and Hamberg~\cite{H1, H2}. A similar approach for Hamiltonian controlled systems was introduced and further studied in the work of Blankenstein, Ortega, van der Schaft, Maschke, Spong, and their collaborators (see, e.g., \cite{MaOrSc2000, OrSpGoBl2002} and related references). The two methods were shown to be equivalent in \cite{ChBlLeMa2002} and a nonholonomic version was developed in \cite{ZBM3,ZBM2002}, and \cite{Bl2003}.

In the controlled Lagrangian approach, one considers a mechanical system with an uncontrolled (free) Lagrangian equal to kinetic energy minus potential energy. To start with, one considers the case in which the Lagrangian is invariant with respect to the action of a Lie group $G$ on the configuration space. To stabilize a relative equilibrium of interest, the kinetic energy is modified to produce a \emph{controlled Lagrangian} which describes the dynamics of the controlled closed-loop system. The equations corresponding to this controlled Lagrangian are the closed-loop equations and the new terms appearing in those equations corresponding to the directly controlled variables correspond to control inputs. The modifications to the Lagrangian are chosen so that no
new terms appear in the equations corresponding to the variables that are not directly controlled. This process of obtaining controlled Euler--Lagrange equations by modifying the original Lagrangian is referred to as \emph{kinetic matching}.

One advantage of this approach is that once the form of the control law is derived using the controlled Lagrangian, the stability of a relative equilibrium of the closed-loop system can be determined by energy methods, using any available freedom in the choice of the parameters of the controlled Lagrangian. To obtain asymptotic stabilization, dissipation-emulating terms are added to the control input.

The method is extended in \cite{BCLM} to the class of Lagrangian mechanical systems with potential energy that may break symmetry, \emph{i.e.}, there is still a symmetry group $G$ for the kinetic energy of the system but one may now have a potential energy that need not be $G$-invariant. Further, in order to define the controlled Lagrangian, a modification to the potential energy is introduced that also breaks symmetry in the group variables. After adding the dissipation-emulating terms to the control input, this procedure allows one to achieve complete state-space asymptotic stabilization of an equilibrium of interest.

The main objective of this paper is to develop the method of controlled Lagrangians for discrete mechanical systems.  The discretization is done in the spirit of discrete variational mechanics, as in \cite{MaWe2001}. In particular, as the closed loop dynamics of a controlled Lagrangian system is itself Lagrangian, it is natural to adopt a variational discretization that exhibits good long-time numerical stability. This study is also motivated by the recent development of structure-preserving algorithms for the numerical simulation of discrete controlled systems, such as recent work on discrete optimization, such as in \cite{GuBl2005, JuMaOb2005,JuMaOb2006}. 
    
The matching procedure is carried out explicitly for discrete systems with one shape and one group degree of freedom to avoid technical issues and to concentrate on the new phenomena that emerge in the discrete setting that have not been observed in the continuous-time theory. In particular, it leads one to either carefully select the momentum levels or introduce a new term in the controlled Lagrangian to perform the discrete kinetic matching. Further, when the potential shaping is carried out, it is necessary to introduce non-conservative forcing in the shape equation associated with the controlled Lagrangian.

It is also shown that once energetically stabilized, the (relative) equilibria  of interest  can be asymptotically stabilized by adding dissipation emulating terms. The separation of controlled dissipation from physical dissipation remains an interesting topic for future research; even in the continuous theory there are interesting questions remaining, as discussed in \cite{WoReBlChLeMa2004}.

The theoretical analysis is validated by simulating the discrete cart-pendulum system on an incline. When dissipation is added, the inverted pendulum configuration is seen to be asymptotically stabilized, as predicted.

The discrete controlled dynamics is used to construct a real-time model predictive controller with piecewise constant control inputs. This serves to illustrate how discrete mechanics can be naturally applied to yield digital controllers for mechanical systems.

The paper is organized as follows: In Sections \ref{discrete_mech.sec} and \ref{matching.sec} we review discrete mechanics and the method of controlled Lagrangians for stabilization of equilibria of mechanical systems. The discrete version of the potential shaping procedure and related stability analysis are discussed in Section~\ref{discrete_matching.sec}. The theory is illustrated with the discrete cart-pendulum system in Section \ref{disc_cart_pendulum.sec}. Simulations and the construction of the digital controller are presented in Sections \ref{simulations.sec} and \ref{digital.sec}.

In a future publication we intend to treat discrete systems with nonabelian symmetries as well as systems with nonholonomic constraints. 

\section{An Overview of Discrete Mechanics}
\label{discrete_mech.sec}
A discrete analogue of Lagrangian mechanics can be obtained by
considering a discretization of Hamilton's principle; this approach
underlies the construction of variational integrators. See Marsden and
West~\cite{MaWe2001}, and references therein, for a more detailed
discussion of discrete mechanics.

Consider a Lagrangian mechanical system with configuration manifold $Q$ and Lagrangian $L : TQ \rightarrow \mathbb{R}$. A key notion is that of a {\em discrete Lagrangian}, which is a map $L^d: Q \times Q \rightarrow \mathbb{R}$ that approximates the action integral along an exact solution of the Euler--Lagrange equations joining the configurations $q_k, q_{k+1} \in Q$,
%-----------------------------
\begin{equation}
\label{exact_ld}
L^d(q_k,q_{k+1})\approx \ext_{q\in\mathcal{C}([0,h],Q)} \int_0^h
L(q,\dot q)\,dt,
\end{equation}
%-----------------------------
where $\mathcal{C}([0,h],Q)$ is the space of curves
$q:[0,h]\rightarrow Q$ with $q(0)=q_k$, $q(h)=q_{k+1}$, and $\ext$
denotes extremum. 

In the discrete setting, the action integral of
Lagrangian mechanics is replaced by an action sum
\begin{equation*}
       S ^d ( q _0, q _1, \dots, q _N) = \sum_{k=0}^{N-1}L^d (q_k, q_{k+1}),
\end{equation*}
where $q_k\in Q$, $ k = 0, 1, \dots, N $, 
is a finite sequence of points in the configuration space. 
The equations are obtained by the discrete Hamilton 
principle, which extremizes the discrete action given
fixed endpoints $q_0$ and $q_N$. Taking the extremum over $q_1,\dots
,q_{N-1}$ gives the {\em discrete Euler--Lagrange equations}
%-----------------------------
\begin{equation*}
         \label{DEL}
         D_1 L^d (q_k, q_{k+1}) + D_2 L^d (q_{k-1}, q_k) = 0,
\end{equation*}
%-----------------------------
for $ k = 1,\dots ,N-1$. This implicitly defines the update map
$\Phi:Q\times Q\rightarrow Q\times Q$, where
$\Phi(q_{k-1},q_k)=(q_k,q_{k+1})$ and $ Q \times Q $ replaces the phase space $ TQ $ of Lagrangian mechanics. 

Since we are concerned with control, we need to consider the effect of
external forces on Lagrangian systems. In the context of discrete
mechanics, this is addressed by introducing the {\em discrete
  Lagrange--d'Alembert principle} (see Kane, Marsden, Ortiz, and 
West~\cite{KaMaOrWe2000}), which states that
\[
\delta\sum_{k=0}^{n-1}L^{d}\left(  q_{k},q_{k+1}\right)  +\sum_{k=0}%
^{n-1}F^{d}\left( q_{k},q_{k+1}\right) \cdot\left( \delta q_{k},\delta
     q_{k+1}\right) =0
\]
for all variations $\bm{\delta q}$ of $\bm{q}$ that vanish at
the endpoints. Here, $\bm{q}$ denotes the vector of positions
$(q_0,q_1,\ldots,q_N)$, and $\bm{\delta q}
=(\delta q_0, \delta
q_1,\ldots, \delta q_N)$, where 
$\delta q_k\in T_{q_k} Q$.
The discrete one-form $F^{d}$ on $Q\times Q$ approximates the impulse
integral between the points $q_k$ and $q_{k+1}$, just as the discrete
Lagrangian $L^d$ approximates the action integral. We define 
the maps
$F^{d}_{1},F^{d}_{2}:Q\times Q\rightarrow T^{\ast}Q$ by the
relations%
\begin{align*}
F^{d}_{2}\left(  q_{0},q_{1}\right) 
\delta q_{1}
& := F^{d}\left(  q_{0},q_{1}\right) \cdot \left(  0,\delta
q_{1}\right),
\\
F^{d}_{1}\left(  q_{0},q_{1}\right) 
\delta q_{0}
& := F^{d}\left(  q_{0},q_{1}\right) \cdot \left(  \delta q_{0},0\right)  .
\end{align*}
The discrete Lagrange--d'Alembert principle may then be rewritten as%
\begin{multline*}
\delta\sum_{k=0}^{n-1}L^{d}\left(  q_{k},q_{k+1}\right) 
\\
+\sum_{k=0}%
^{n-1}\left[  F^{d}_{1}\left(  q_{k},q_{k+1}\right)  
\delta q_{k}%
+ F^{d}_{2}\left(  q_{k},q_{k+1}\right)  
\delta q_{k+1}\right]
= 0
\end{multline*}
for all variations $\bm{\delta q}$ of $\bm{q}$ that vanish at the endpoints. This is equivalent to the {\em forced discrete Euler--Lagrange equations}
\begin{align*}
D_{1}L^{d}\left( q_{k}
,q_{k+1}\right) &+ D_{2}L^{d}\left(  q_{k-1},q_{k}\right) 
\\
&+ F^{d}_{1}\left(  q_{k},q_{k+1}\right)  + F^{d}_{2}\left(
q_{k-1},q_{k}\right)  =0.
\end{align*}

\section{Matching and Controlled Lagrangians}\label{matching.sec}
\subsection{Controlled Euler--Lagrange Equations}
This paper focuses on systems with one shape and one group degree of freedom. It is further assumed that the configuration space $Q$ is the direct product of a one-dimensional shape space $S$ and a one-dimensional Lie group $G$. 

The configuration variables are written as $ q = (\phi, s) $, with $ \phi \in S $, and $ s \in G $. The velocity phase space, $TQ$, has coordinates $(\phi, s, \dot{\phi}, \dot{s})$. 
The Lagrangian is the kinetic minus potential energy
\begin{equation}\label{continuous_11_lagr.eqn}
L(q, \dot{q}) = \tfrac 12 \big[
	\alpha \dot \phi ^2 + 2 \beta (\phi) \dot \phi \dot s 
	+ \gamma \dot s ^2 
\big] - V (q),
\end{equation}
with $G$-invariant kinetic energy. 
The corresponding controlled Euler--Lagrange dynamics is 
\begin{align}\label{EL_L_1.eqn} 
\frac{d}{dt} \frac{\partial L}{\partial \dot \phi} 
- \frac{\partial L}{\partial \phi} &= 0,
\\
\label{EL_L_2.eqn}
\frac{d}{dt} \frac{\partial L}{\partial \dot{s}} &= u,
\end{align} 
where $u$ is the control input. 
\subsection{Continuous-Time Kinetic Shaping}
Assume that the potential energy is $G$-invariant, \emph{i.e.,} $ V (q) = V (\phi) $, and that the \emph{relative equilibria} $ \phi = \phi _e $, $ \dot{s} = \text{const} $ are unstable and given by non-degenerate critical points of $ V (\phi) $. To stabilize the relative equilibria $ \phi = \phi _e $, $ \dot{s} = \text{const} $ with respect to $\phi$, kinetic shaping is used. The controlled Lagrangian in this case is defined by
\begin{equation}\label{cont_controlled_lagrangian_11.eqn}
L _{\tau,\sigma}(q, \dot{q}) = 
L( \phi, \dot \phi, \dot{s} + \tau (\phi) \dot \phi) +
{\textstyle \frac12} \sigma \gamma (\tau (\phi) \dot \phi ) ^2, 
\end{equation} 
where $ \tau (\phi) = \kappa \beta (\phi) $. This velocity shift corresponds to a new choice of the horizontal space (see \cite{BLM3} for details). The dynamics is just the Euler--Lagrange dynamics for controlled Lagrangian \eqref{cont_controlled_lagrangian_11.eqn},
\begin{align}\label{EL_CL_1.eqn}
\frac{d}{dt} \frac{\partial L _{\tau, \sigma}}{\partial \dot \phi} 
- \frac{\partial L_{\tau, \sigma}}{\partial \phi} &= 0,
\\
\label{EL_CL_2.eqn}
\frac{d}{dt} \frac{\partial L_{\tau, \sigma}}{\partial \dot{s}} &= 0.
\end{align}
Lagrangian \eqref{cont_controlled_lagrangian_11.eqn} satisfies the simplified matching conditions of \cite{BCLM} when the kinetic energy metric coefficient $\gamma$ in \eqref{continuous_11_lagr.eqn} is constant. 

Setting 
\(
u = - d \big( \gamma \tau (\phi) \dot{\phi} \big)/ dt 
\)
defines the control input, makes equations \eqref{EL_L_2.eqn} and \eqref{EL_CL_2.eqn} identical, and results in controlled momentum conservation by dynamics \eqref{EL_L_1.eqn} and \eqref{EL_L_2.eqn}. Setting $ \sigma = - 1/\gamma \kappa $ makes equations \eqref{EL_L_1.eqn} and \eqref{EL_CL_1.eqn} reduced on the controlled momentum level identical. 

A very interesting feature of systems \eqref{EL_L_1.eqn}, \eqref{EL_L_2.eqn} and \eqref{EL_CL_1.eqn}, \eqref{EL_CL_2.eqn} is that the \emph{reduced} dynamics are the same on all momentum levels, which follows from the independence of equations \eqref{EL_L_1.eqn} and \eqref{EL_CL_1.eqn} of the group velocity $ \dot{s} $. We will see in Section \ref{discrete_matching.sec} that this property does not hold in the discrete setting, and one has to carefully select the momentum levels when performing discrete kinetic shaping.

\subsection{Continuous-Time Potential Shaping}
Now, consider the case when the kinetic energy is group invariant, but the potential is not. Consider the special case when the potential energy is $ V (q) = V _1 (\phi) + V _2 (s) $ with $ V _1 (\phi) $ having a local non-degenerate maximum at $ \phi _e $, and the goal is to stabilize the \emph{equilibrium} $ \phi = \phi _e $, $ s = s _e $. As it becomes necessary to shape the potential energy as well, the controlled Lagrangian is defined by the formula
\multlinegap=0em
\begin{multline}\label{cont_potential_controlled_lagrangian_11.eqn}
L _{\tau,\sigma,\rho,\epsilon } 
( \phi, s, \dot \phi, \dot{s} ) = 
L( \phi, s, \dot \phi, \dot{s} + \tau (\phi) \dot \phi) 
+
{\textstyle \frac12} \sigma \gamma 
(\tau (\phi) \dot \phi ) ^2 
\\
+ {\textstyle \frac12}(\rho - 1)\gamma (\dot{s} + (\sigma - 1) \tau (\phi) \dot \phi ) ^2 
+ V _2 (s) - V _\epsilon (y) ,\!\!
\end{multline}
where 
\begin{equation} \label{ydef}
y = s - \int _{\phi _e } ^\phi \frac{1}{\gamma}\left(
\frac{1}{\sigma} - \frac{\rho -1}{\rho} 
\right) \beta (z)\, d z,
\end{equation}
$ V _\epsilon (y) $ is an arbitrary negative-definite function, and $ (\phi _e, s _e) $ is the equilibrium of interest. 

Below we assume that $ \phi _e = 0 $, which can be always accomplished by an appropriate choice of local coordinates for each (relative) equilibrium. 

\section{Discrete Shaping}
\label{discrete_matching.sec}

\subsection{Discrete Controlled Dynamics}
\label{disc_controlled_dynamics.sec} 

In discretizing the method of controlled Lagrangians, we 
combine formulae \eqref{exact_ld}, \eqref{continuous_11_lagr.eqn}, and \eqref{cont_controlled_lagrangian_11.eqn}.
In the rest of this paper, we will adopt the notations
%-----------------------------
\begin{equation*}\label{notations.eqn}
q _{k+1/2} = \frac {q _k + q _{k+1}}{2}, \quad 
\Delta q _k = q _{k+1} - q _k, \quad q _k = (\phi _k, s _k).
\end{equation*} 
This allows us to construct a \emph{second-order accurate} 
discrete Lagrangian 
\[
L ^d (q _k,q  _{k+1}) = 
h L( 
q _{k+1/2} , \Delta q _k /h 
).
\]
Thus, for a system with one shape and one group degree of freedom the discrete Lagrangian is given by the formula
\begin{multline}
\label{discrete_second_order_lagrangian}
L ^d (q _k,q  _{k+1}) 
= \frac {h}{2} \bigg[
	\alpha \Big( \frac {\Delta \phi _k}{h} \Big) ^2 
	+
	2 \beta ( \phi _{k+1/2} ) \frac{\Delta \phi _k}{h} 
	\frac{\Delta s _k}{h}  
\\
	+
	\gamma \Big( \frac{\Delta s _k}{h} \Big) ^2 
	\bigg] - h V (q _{k+1/2}).\!\!
\end{multline}
The discrete dynamics is governed by the equations
\begin{align} 
\label{dcp1.eqn}
\pder{L ^d (q _k,q  _{k+1})}{\phi _k} + 
\pder{L ^d (q_{k-1},q  _k)}{\phi _k} &= 0,
\\
\label{dcp2.eqn}
\pder{L ^d (q _k,q  _{k+1})}{s _k} + 
\pder{L ^d (q _{k-1},q  _k)}{s _k} &= -u _k ,
\end{align}
where $ u _k $ is the control input. 

\subsection{Kinetic Shaping}\label{kinetic_shaping.sec} 
At first, it will be  assumed that the potential energy is $G$-invariant, \emph{i.e.,} $ V (q) = V (\phi) $, and that relative equilibria $ \phi _k = 0, \Delta s _k = \text{const} $ of \eqref {dcp1.eqn} and \eqref {dcp2.eqn} in the absence of control input are unstable. We will see that one needs to either appropriately select the momentum levels or introduce a new parameter into the controlled Lagrangian to complete the matching procedure.

Motivated by the continuous-time matching procedure (see Section \ref{matching.sec}), we define the discrete controlled Lagrangian by the formula
\begin{multline} 
\label{disc_controlled_lagrangian_11.eqn}
L ^d _{\tau,\sigma} (q _k, q _{k+1}) = h L _{\tau,\sigma}(q _{k+1/2}, \Delta q _k /h)
\\
=
h \Big[
L \big( \phi _{k+\frac12} , \Delta \phi _k /h  , \Delta s _k /h 
+ \kappa \beta ( \phi _{k+\frac12} ) \Delta \phi _k /h \big) 
\\
+ \frac{\sigma \gamma}{2} \Big( 
\kappa \beta ( \phi _{k+\frac12} ) \Delta \phi _k /h \Big) ^2  
\Big].\!\!\!
\end{multline}
where $ L _{\tau,\sigma}(q, \dot{q}) $ is the continuous-time controlled Lagrangian. 
The dynamics associated with \eqref{disc_controlled_lagrangian_11.eqn} is
\begin{align} 
\label{controlled_cart_pendulum_1.eqn}
\pder{L ^d  _{ \tau , \sigma } (q _k,q  _{k+1})}{\phi _k} + 
\pder{L ^d  _{ \tau , \sigma } (q_{k-1},q  _k)}{\phi _k} &= 0,
\\
\label{controlled_cart_pendulum_2.eqn}
\pder{L ^d  _{ \tau , \sigma } (q _k,q  _{k+1})}{s _k} + 
\pder{L ^d  _{ \tau , \sigma } (q _{k-1},q  _k)}{s _k} &= 0. 
\end{align}
Equation \eqref{controlled_cart_pendulum_2.eqn} is equivalent to the
\emph{discrete controlled momentum conservation:}
\begin{equation}\label{controlled_momentum_level.eqn}
p _k = \mu ,
\end{equation} 
where 
\begin{align}\label{discrete_cp_momentum.eqn}
\nonumber
p _k &= - \frac{\partial }{\partial s _k} L ^d _{\tau,\sigma} (q _k, q _{k+1})
\\
&=  \frac{
	(1 + \gamma \kappa) \beta (\phi _{k+1/2} ) 
	\Delta \phi _k + \gamma \Delta s _k} {h}.
\end{align} 

Setting
\begin{equation}\label{discrete_u.eqn}
u _k = - \frac{\gamma \Delta \phi _k \tau ( \phi _{k+1/2} ) 
- \gamma \Delta \phi _{k-1} \tau ( \phi _{k-1/2})}{h }
\end{equation} 
makes equations \eqref{dcp2.eqn} and \eqref{controlled_cart_pendulum_2.eqn}
identical and allows 
one to represent the discrete momentum equation
\eqref{dcp2.eqn} as the discrete momentum conservation law 
\begin{equation}\label{momentum_level.eqn}
p _k = p.
\end{equation} 

\begin{theorem} \label{discrete_matching.thm}
\emph{
The dynamics determined by equations \eqref{dcp1.eqn} and \eqref{dcp2.eqn} restricted to the momentum level $ p _k = p $ is equivalent to the dynamics of equations
\eqref{controlled_cart_pendulum_1.eqn} and
\eqref{controlled_cart_pendulum_2.eqn} restricted to the momentum level $
p _k = \mu $ if and only if the matching conditions 
\begin{equation}\label{kinetic_matching.eqn}
\sigma 
= - \frac{1}{\gamma \kappa}, \qquad 
\mu = \frac{p}{1 + \gamma \kappa}.
\end{equation} hold.
}
\end{theorem}

\begin{proof}
Solve equations \eqref{controlled_momentum_level.eqn} and
\eqref{momentum_level.eqn} for $\Delta s _k $ and
substitute the solutions in equations \eqref{dcp1.eqn} and
\eqref{controlled_cart_pendulum_1.eqn}, respectively. This process is a simple version of discrete reduction \cite{JaLeMaWe2005}. A~computation shows that the equations obtained this way are equivalent if and only if 
\begin{multline}\label{disc_matching_cnd_1.eqn} 
h \bigg[
\frac{\mu - p + \gamma \kappa \mu}{\gamma}
\frac{\partial }{\partial \phi _k } 
	\bigg(
		\beta (\phi _{k+1/2}) \frac{\Delta \phi _k}{h} + 
		\beta (\phi _{k-1/2}) \frac{\Delta \phi _{k-1}}{h} 
	\bigg)
\\
+
\frac{\kappa + \gamma \sigma \kappa ^2}{2}
	\frac{\partial }{\partial \phi _k } 
	\bigg(
		\beta ^2 (\phi _{k+1/2}) \Big(\frac{\Delta \phi _k}{h}\Big) ^2 
\\ 
	+ \beta ^2 (\phi _{k-1/2}) \Big(\frac{\Delta \phi _{k-1}}{h}\Big) ^2  
	\bigg) \bigg] = 0. 
\end{multline}
Since $ \beta (\phi) \neq 0 $ and $ \Delta \phi _k \neq 0 $ generically, equations \eqref{dcp1.eqn} and \eqref{dcp2.eqn} are equivalent if and only if 
\[
\mu - p + \gamma \kappa \mu = 0,
\qquad 
\kappa + \gamma \sigma \kappa ^2,
\]
which is equivalent to \eqref{kinetic_matching.eqn}
Note that the momentum levels
$p$ and $\mu$ {\em are not} the same.
\end{proof}
\medskip

\noindent \emph{Remark.\ } 
As $ h \to 0 $, formulae \eqref{discrete_u.eqn} and \eqref{disc_matching_cnd_1.eqn} become
\[
u = - \frac{d}{dt} \big(\gamma \tau (\phi) \dot{\phi} \big) 
\]
and 
\[
- (\kappa + \gamma \sigma \kappa ^2) 
(\beta ^2 (\phi) \ddot \phi + \beta (\phi) \beta ' (\phi) \dot{\phi} ^2)
= 0,
\]
respectively. That is, as $ h \to 0 $, one recovers the continuous-time control input and the continuous-time matching condition, $ \sigma = - 1/\gamma \kappa $. 
Condition $ \mu = p/(1 + \gamma \kappa) $ becomes redundant after taking the limit, \emph{i.e.,} the reduced dynamics can be matched on arbitrary momentum levels in the continuous-time case, which agrees with observations made in Section \ref{matching.sec}. 
 
We now discuss an alternative matching procedure. Define the
discrete controlled Lagrangian $\Lambda  ^d _{ \tau , \sigma, \lambda } ( q _k , q _{k+1}) $  by the formula 
\begin{multline*} 
h \Big[
L \big( \phi _{k+\frac12} , \Delta \phi _k /h  , \Delta s _k /h 
+ \kappa \beta ( \phi _{k+\frac12} ) \Delta \phi _k /h \big) 
\\
+ \frac{\sigma \gamma}{2} \Big( 
\kappa \beta ( \phi _{k+\frac12} ) \Delta \phi _k /h \Big) ^2  
+ \lambda \kappa \beta ( \phi _{k+\frac12} ) \Delta \phi _k /h 
\Big].
\end{multline*} 
The discrete dynamics associated with this Lagrangian is
\begin{align} 
\label{alternative_controlled_cart_pendulum_1.eqn}
\pder{\Lambda ^d  _{ \tau , \sigma, \lambda } (q _k,q  _{k+1})}{\phi _k} + 
\pder{\Lambda ^d  _{ \tau , \sigma, \lambda } (q_{k-1},q  _k)}{\phi _k} &= 0,
\\
\label{alternative_controlled_cart_pendulum_2.eqn}
\pder{\Lambda ^d  _{ \tau , \sigma } (q _k,q  _{k+1})}{s _k} + 
\pder{\Lambda ^d  _{ \tau , \sigma } (q _{k-1},q  _k)}{s _k} &= 0. 
\end{align}
The discrete controlled momentum is given by formula
\begin{align}\label{alt_discrete_controlled_momentum.eqn}
\nonumber
p _k &= - \frac{\partial }{\partial s _k} \Lambda ^d _{\tau,\sigma,\lambda} (q _k, q _{k+1}) 
\\
&=
\frac{(1 + \gamma \kappa) \beta (\phi _{k+1/2} ) 
	\Delta \phi _k + \gamma \Delta s _k}{h}
\end{align} 
and equation
\eqref{alternative_controlled_cart_pendulum_2.eqn} is equivalent to
the discrete momentum conservation \eqref{momentum_level.eqn}.
\begin{theorem}\label{alternative_discrete_matching.thm}
{\em The dynamics \eqref{dcp1.eqn} and \eqref{dcp2.eqn} restricted to the
momentum level $ p _k = p $ is equivalent to the dynamics
\eqref{alternative_controlled_cart_pendulum_1.eqn} and
\eqref{alternative_controlled_cart_pendulum_2.eqn} restricted to the
same momentum level if and only if the matching conditions
\begin{equation}\label{alternative_matching.eqn}
\sigma = - \frac{1}{\gamma \kappa}, \qquad 
\lambda  = - p.
\end{equation} 
hold.
}
\end{theorem}
\begin{proof}
Similar to the proof of Theorem \ref{discrete_matching.thm}, solve equation 
\eqref{momentum_level.eqn} for $\Delta s _k $ and
substitute the solution in equations \eqref{dcp1.eqn} and
\eqref{alternative_controlled_cart_pendulum_1.eqn}, respectively. 
A computation shows that the equations obtained this way are equivalent if and only if 
\begin{multline}\label{disc_matching_cnd_2.eqn} 
h \bigg[
(\kappa p + \kappa \lambda)
\frac{\partial }{\partial \phi _k } 
	\bigg(
		\beta (\phi _{k+1/2}) \frac{\Delta \phi _k}{h} + 
		\beta (\phi _{k-1/2}) \frac{\Delta \phi _{k-1}}{h} 
	\bigg)
\\
+
\frac{\kappa + \gamma \sigma \kappa ^2}{2}
	\frac{\partial }{\partial \phi _k } 
	\bigg(
		\beta ^2 (\phi _{k+1/2}) \Big(\frac{\Delta \phi _k}{h}\Big) ^2 
\\ 
	+ \beta ^2 (\phi _{k-1/2}) \Big(\frac{\Delta \phi _{k-1}}{h}\Big) ^2  
	\bigg) \bigg] = 0,\!\!
\end{multline}
which implies \eqref{alternative_matching.eqn}. 
Note that in this case we add an extra term to the controlled
Lagrangian which eliminates the need for adjusting the momentum level.
\end{proof}
\medskip

\noindent
\emph{Remark.\ } The ratio $ \Lambda ^d _{\tau, \sigma, \lambda} / h $ becomes $ L ^d _{\tau,\sigma} + \lambda \kappa \beta (\varphi) \dot \varphi $ as $ h \to 0 $. That is, as we let the time step go to $0$, we obtain the continuous-time controlled Lagrangian modified by a term which is a derivative of the function $ \lambda \kappa \int \beta (\phi) \, d \phi $ with respect to time. It is well-known that adding such a derivative term to a Lagrangian does not change the dynamics associated with this Lagrangian. 

The stability properties of the relative equilibria $
\phi _k = 0 $, $ s _k = \text{const} $ of equations \eqref{dcp1.eqn} and \eqref{dcp2.eqn} are now investigated. 
\medskip

\begin{theorem}
{\em
The relative equilibria $ \phi _k = 0 $, $ \Delta s _k = \text{const} $ of equations \eqref{dcp1.eqn} and \eqref{dcp2.eqn}, with $ u _k $
defined by \eqref{discrete_u.eqn}, are \textbf{spectrally stable} if 
\begin{equation}\label{discrete_spectrum.eqn}
\kappa > \frac{\alpha \gamma - \beta ^2 (0)}{\beta ^2 (0)\, \gamma}.
\end{equation} 
}
\end{theorem}
\medskip

\begin{proof}
Let $ V'' (0) = - C $, where $ C > 0 $ (see Section~\ref{matching.sec}). 
The linearization of the reduced dynamics \eqref{dcp1.eqn}
and \eqref{dcp2.eqn}  at $ \phi = 0 $ is computed to be
\begin{multline}\label{linearized_discrete_dynamics.eqn}
\frac{\alpha \gamma - \beta ^2 (0) - \beta ^2 (0) \gamma \kappa}{h ^2 \gamma} \left(
        \Delta \phi _{k-1} - \Delta \phi _k 
\right) 
\\
+ \frac C4 \left(
\phi _{k-1} + 2 \phi _k + \phi _{k+1}
\right) = 0.
\end{multline}
Observe that the value of $p$ does not affect the linearized dynamics. 

The linearized dynamics preserves the quadratic approximation of the
discrete energy
\begin{equation}\label{quadratic_energy.eqn}
\frac{\alpha \gamma - \beta ^2 (0) - \beta ^2 (0) \gamma \kappa}{2 \gamma} 
\bigg(\frac{\Delta \phi _k}{h ^2}\bigg)^2 - \frac {C}{2} \phi _{k+1/2} ^2.
\end{equation} 
The equilibrium $ \phi _k = 0 $ of
\eqref{linearized_discrete_dynamics.eqn} is stable if and only if the
function \eqref{quadratic_energy.eqn} is negative-definite at $ \phi
_k = \phi _{k+1} = 0 $. The latter requirement is equivalent to
condition \eqref{discrete_spectrum.eqn}.
\end{proof}
\medskip

\noindent
\emph{Remark.\ }
The stability condition \eqref{discrete_spectrum.eqn} is identical to the
stability condition of the continuous-time cart-pendulum system, and it can be rewritten as 
\[
-\frac{\beta ^2 (0)}{\alpha \gamma - \beta ^2 (0)} < \sigma < 0 .
\]

The spectrum of the linear map $ (\phi _{k-1}, \phi
_k)\mapsto (\phi _k, \phi _{k+1}) $ defined by
\eqref{linearized_discrete_dynamics.eqn} belongs to the unit
circle. Spectral stability in this situation is not sufficient to conclude
nonlinear stability. 

We now modify the control input \eqref{discrete_u.eqn} by adding the
{\em kinetic discrete dissipation-emulating term} 
\begin{equation*}
\label{frictionk.eqn}
\frac{D (\Delta \phi _{k-1} + \Delta \phi _k)}{2 h}
\end{equation*} 
in order to achieve the asymptotic stabilization of the upward position
of the pendulum. In the above, $D$ is a positive constant. The discrete
momentum conservation law becomes
\begin{equation*}\label{modified_momentum_conservation.eqn}
p _k  - \frac{D \, \phi _{k+1/2}}{h} = p.
\end{equation*} 
Straightforward calculation shows that the spectrum of the matrix of the linear map $
(\phi _{k-1}, \phi _k)\mapsto (\phi _k, \phi _{k+1}) $ defined
by the reduced discrete dynamics belongs to the open unit disc. This
implies that the equilibrium $ \phi = 0 $ is asymptotically stable.

\subsection{Potential Shaping}\label{potential_shaping.sec} 
Recall that the 
the discrete dynamics associated with discrete Lagrangian \eqref{discrete_second_order_lagrangian} is governed by equations
\eqref{dcp1.eqn}~and~\eqref{dcp2.eqn}, where $ u _k $ is the control input. The goal of the procedure developed in this section is to stabilize the equilibrium $ (\phi , s) = (0, 0) $ of \eqref{dcp1.eqn} and \eqref{dcp2.eqn}.

Motivated by \eqref{cont_potential_controlled_lagrangian_11.eqn}, we define 
the second-order accurate discrete
controlled Lagrangian by the formula 
\begin{equation}
\label{discrete_second_order_controlled_lagrangian}
L^d_{\tau,\sigma,\rho,\epsilon}(q_k,q_{k+1}) = 
h L_{\tau,\sigma,\rho,\epsilon}( 
q _{k+1/2} , \Delta q _k /h 
) ,
\end{equation}
where $ q _k = ( \phi _k, s _k) $.

The dynamics associated with \eqref{discrete_second_order_controlled_lagrangian} is amended by the term $ w _k $ in the discrete shape equation:
\begin{align} 
\label{controlled_lagrangian_dynamics_1.eqn}
\pder{L ^d  _{ \tau,\sigma,\rho,\epsilon} (q _k,q  _{k+1})}{\phi _k} + 
\pder{L ^d  _{ \tau , \sigma,\rho,\epsilon } (q_{k-1},q  _k)}{\phi _k} &= -w _k ,
\\
\label{controlled_lagrangian_dynamics_2.eqn}
\pder{L ^d  _{ \tau , \sigma,\rho,\epsilon } (q _k,q  _{k+1})}{s _k} + 
\pder{L ^d  _{ \tau , \sigma,\rho,\epsilon } (q _{k-1},q  _k)}{s _k} &= 0. 
\end{align}
This term $ w _k $ is important for matching systems \eqref{dcp1.eqn}, \eqref{dcp2.eqn} and \eqref{controlled_lagrangian_dynamics_1.eqn}, \eqref{controlled_lagrangian_dynamics_2.eqn}. \emph{The presence of the terms $w _k$ represents an interesting (but manageable) departure from the continuous theory.}
Let 
\[
J _k = \rho \gamma 
\big(
{\Delta s _k}/{h} 
- (\sigma - 1) 
\tau (\phi _{k+\frac12})
{\Delta \phi _k}/{h}
\big).
\]
The following statement is proved by a straightforward calculation:
\begin{theorem}\label{discrete_potential_matching.thm}
\emph{
The dynamics \eqref{dcp1.eqn}, \eqref{dcp2.eqn} is equivalent to the dynamics
\eqref{controlled_lagrangian_dynamics_1.eqn},
\eqref{controlled_lagrangian_dynamics_2.eqn} if and only if $u _k$ and $w _k$ are given by
}
\begin{align}\label{discrete_u_potential.eqn}
\nonumber
u _k &= 
 \frac{h}{2} \left[
V _2 ' (s _{k+\frac12}) + V _2 ' (s _{k-\frac12})
\right]
\\
\nonumber
& \quad \,
- \frac{h}{2 \rho} \left[
V _\epsilon  ' (s _{k+\frac12}) + V _\epsilon ' (s _{k-\frac12})
\right]
\\
& \quad \, -
\frac{\gamma \Delta \phi _k \tau ( \phi _{k+1/2} ) 
- \gamma \Delta \phi _{k-1} \tau ( \phi _{k-1/2})}{h },
\end{align}
and
\begin{align*} 
w _k &= 
\Big(
1 - \sigma  + \frac{\sigma}{\rho}
\Big)
\Big(
\tau ( \phi _{k+\frac12})
\Big[ - \gamma \rho J _k + 
\frac{h}{2} V ' _\epsilon (y _{k+\frac12})
\Big]
\\
& \quad \, 
+ \tau ( \phi _{k-\frac12})
\Big[ \gamma \rho J _{k-1} + 
\frac{h}{2} V ' _\epsilon (y _{k-\frac12})
\Big]
\\
& \quad \, 
- \tau ' ( \phi _{k+\frac12}) J _k \Delta \phi _k 
- \tau ' ( \phi _{k-\frac12}) J _{k - 1} 
\Delta \phi _{k - 1}
\Big) , 
\end{align*} 
where $y _k$ is obtained by substituting $\phi_k$ and $s_k$ in formula \eqref{ydef}.
\end{theorem}

\smallskip

\noindent
\emph{Remark.\ }
Equations \eqref{dcp1.eqn}, \eqref{dcp2.eqn} define closed-loop dynamics when $ u _k $ is given by formula \eqref{discrete_u_potential.eqn}.  
The terms $ w _k $ vanish when $ \beta (\phi) = \text{const} $ as they become proportional to the left-hand side of equation \eqref{controlled_lagrangian_dynamics_2.eqn}. 

As in the case of kinetic shaping, the stability analysis is done by means of an analysis of the spectrum of the linearized discrete equations. We assume that the equilibrium to be stabilized is $ (\phi _k, s _k) = (0,0) $.  

\begin{theorem}\label{linear_potential_stability.thm} 
\emph{
The equilibrium $ (\phi _k , s _k) = (0, 0) $ of equations \eqref{controlled_lagrangian_dynamics_1.eqn} and \eqref{controlled_lagrangian_dynamics_2.eqn} is \textbf{spectrally stable} if 
\begin{equation}\label{discrete_spectrum_potential.eqn}
- \frac{\beta ^2 (0)}{\alpha \gamma - \beta ^2 (0)} < \sigma < 0, \quad \rho < 0, \quad \text{and} \quad V _\epsilon '' (0) < 0.\!
\end{equation}
}
\end{theorem}
\medskip
\begin{proof}
The linearized discrete equations are 
\begin{align} 
\label{linearized_1.eqn}
\pder{\mathcal L ^d  _{ \tau,\sigma,\rho,\epsilon} (q _k,q  _{k+1})}{\phi _k} + 
\pder{\mathcal L ^d  _{ \tau , \sigma,\rho,\epsilon } (q_{k-1},q  _k)}{\phi _k} &= 0 ,
\\
\label{linearized_2.eqn}
\pder{\mathcal L ^d _{ \tau , \sigma,\rho,\epsilon } (q _k,q  _{k+1})}{s _k} + 
\pder{\mathcal L ^d _{ \tau , \sigma,\rho,\epsilon } (q _{k-1},q  _k)}{s _k} &= 0, 
\end{align}
where $ \mathcal{L} ^d _{ \tau , \sigma,\rho,\epsilon } (q _k,q  _{k+1}) $ is the quadratic approximation of $ L ^d _{ \tau , \sigma,\rho,\epsilon } $ at the equilibrium 
(\emph{i.e.}, $ \beta (\phi) $,  $ V _1 (\phi) $, and $ V _\epsilon (y) $ in $ L ^d _{ \tau , \sigma,\rho,\epsilon } $ are replaced by $ \beta (0) $, $ \frac12 V _1 '' (0) \phi ^2 $, and  $ \frac 12 V _\epsilon '' (0) y ^2 $, respectively). \emph{Note the absence of the term $ w _k $ in equation~\eqref{linearized_1.eqn}. }

The linearized dynamics preserves the quadratic approximation of the
discrete energy $ E _{k, k+1} $ defined by 
\begin{multline}\label{potential_shaping_quadratic_energy.eqn} 
\frac{\alpha \gamma \sigma ^2 - 
\beta (0) ^2 (\sigma-1)(\rho (\sigma-1) - \sigma )}
{2 \gamma \sigma ^2 h} 
\Delta \phi _k^2 
\\
+ \frac{\beta (0) \rho (\sigma - 1)}{\sigma h}
\Delta \phi _k \Delta s _k 
+ \frac{\gamma \rho}{2 h} \Delta s _k ^2 
\\
+ \frac{h}{2} V _1 '' (0) \phi _{k+\frac12} ^2 
+ \frac{h}{2} V _\epsilon '' (0) x _{k+\frac12} ^2 
,
\end{multline} 
where
\begin{equation} \label{xdef} 
x  = s  + 
\left(
\frac{\rho - 1}{\rho} - \frac{1}{\sigma}
\right)
\frac{\beta (0)}{\gamma}\, \phi .
\end{equation}
Since $ V _1 '' (0) $ is negative, 
the equilibrium $ (\phi _k, s _k) = (0,0) $ of equations 
\eqref{linearized_1.eqn} and \eqref{linearized_2.eqn} is stable if 
the quadratic approximation of the discrete controlled energy \eqref {potential_shaping_quadratic_energy.eqn} is negative-definite. The latter requirement is equivalent to
conditions \eqref {discrete_spectrum_potential.eqn}. The spectrum of the linearized discrete dynamics in this case belongs to the unit circle. \end{proof}
\smallskip

\noindent
\emph{Remarks.\ }
Spectral stability in this situation is not sufficient to conclude
nonlinear stability.
The stability conditions \eqref {discrete_spectrum_potential.eqn} are identical to the stability conditions of the corresponding continuous-time system.

Following \cite{BCLM}, we now modify the control input \eqref {discrete_u_potential.eqn} by adding the
\emph{potential discrete dissipation-emulating term} 
\begin{equation}\label{frictionp.eqn}
\frac{D (\Delta y _{k-1} + \Delta y _k)}{2 h}
\end{equation} 
in order to achieve the asymptotic stabilization of the equilibrium $ (\phi _k, s _k) = (0,0) $. In the above, $D$ is a constant. 
The linearized discrete dynamics becomes
\begin{align} 
\nonumber 
\label{linearized_damped_1.eqn}
\pder{\mathcal L ^d  _{ \tau,\sigma,\rho,\epsilon} (q _k,q  _{k+1})}{\phi _k} + 
\pder{\mathcal L ^d  _{ \tau , \sigma,\rho,\epsilon } (q_{k-1},q  _k)}{\phi _k} \qquad \quad &
\\
= - \left( 
\frac{\rho - 1}{\rho} - \frac{1}{\sigma}
\right)
\frac{\beta (0)}{\gamma} 
\frac{D (\Delta x _{k-1} + \Delta x _k)}{2 h} &,
\\
\nonumber
\label{linearized_damped_2.eqn}
\pder{\mathcal L ^d _{ \tau , \sigma,\rho,\epsilon } (q _k,q  _{k+1})}{s _k} + 
\pder{\mathcal L ^d _{ \tau , \sigma,\rho,\epsilon } (q _{k-1},q  _k)}{s _k} \qquad \quad &
\\
= - \frac{D (\Delta x _{k-1} + \Delta x _k)}{2 h} &, 
\end{align}
where $ x_k$ is obtained by substituting $\phi_k$ and $s_k$ in formula \eqref{xdef}.

\begin{theorem} \label{asymtotic_stability.thm}
\emph{
The equilibrium $ (\phi _k, s _k) = (0, 0) $ of equations \eqref{linearized_damped_1.eqn} and \eqref{linearized_damped_2.eqn} is asymptotically stable if conditions \eqref{discrete_spectrum.eqn} are satisfied and $ D $ is positive.
}
\end{theorem} 
\begin{proof}
Multiplying equations \eqref{linearized_damped_1.eqn} and \eqref{linearized_damped_2.eqn} by $ (\Delta \phi _{k-1} + \Delta \phi _k )/2 $ and $ (\Delta s _{k-1} + \Delta s _k )/2 $, respectively, we obtain 
\[
E _{k, k+1 } = E _{k-1, k } + \frac{D h}{2} 
\left(
\frac{\Delta x _{k-1} + \Delta x _k}{2h}
\right) ^2,
\] 
where $ E _{k, k+1} $ is the quadratic approximation of the discrete energy \eqref{quadratic_energy.eqn}. Recall that $ E _{k, k+1} $ is negative-definite (see the proof of Theorem \ref{linear_potential_stability.thm}). It is possible to show that, in some neighborhood of $ (\phi _k, s _k) = (0,0) $,  the quantity $ \Delta x _{k-1} + \Delta x _k \not\equiv 0 $ along a solution of equations \eqref{linearized_damped_1.eqn} and \eqref{linearized_damped_2.eqn} unless this solution is the equilibrium $ (\phi _k, s _k) = (0,0) $. Therefore, $  E _{k, k+1} $ increases along non-equilibrium solutions of \eqref{linearized_damped_1.eqn} and \eqref{linearized_damped_2.eqn}. Since equations \eqref{linearized_damped_1.eqn} and \eqref{linearized_damped_2.eqn} are linear, this is only possible if the spectrum of \eqref{linearized_damped_1.eqn} and \eqref{linearized_damped_2.eqn} is inside the open unit disk, which implies asymptotic stability of the equilibrium of both linear system \eqref{linearized_damped_1.eqn} and \eqref{linearized_damped_2.eqn} and nonlinear system \eqref{dcp1.eqn} and \eqref{dcp2.eqn} with potenital discrete dissipation-emulating term \eqref{frictionp.eqn} added to $ u _k $.
\end{proof}

\section{Stabilization of the Discrete \\ Pendulum on the Cart}
\label{disc_cart_pendulum.sec} 
A basic example treated in earlier papers in the smooth setting is
the {\em pendulum on a cart}.  Let $s$ denote the position of the cart
on the $s$-axis, $\phi$ denote the angle of the pendulum with
the upright vertical, and $\psi$ denote the elevation angle of the incline, as in Figure \ref{cart.figure}.
%-------------------------------------
\begin{figure}[ht]
\begin{center}
\begin{overpic}[width=0.32\textwidth]
{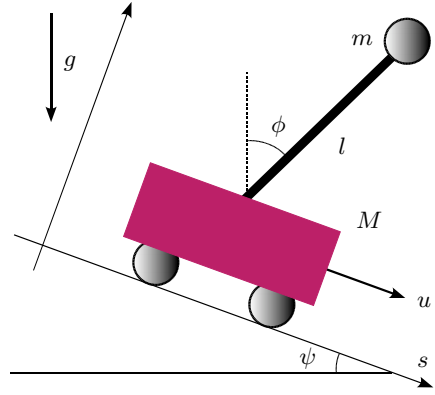}
\put(94,6){\small $s$}
\put(94,19,8){\small $u$}
\put(79,80){\small $m$}
\put(76,55){\small $l$}
\put(80,38){\small $M$}
\put(13,75){\small $g$}
\put(60.5,60){\small $\phi$}
\put(67,6.76){\small $\psi$}
\end{overpic}
\caption{\footnotesize
The pendulum on a cart going down an inclined plane under gravity. The control force is in the direction $s$, the overall motion of the cart.}
\label{cart.figure}
\end{center}
\end{figure}
%-----------------------------
The configuration space for this system is $Q = S \times G = S^1
\times \mathbb{R}$, with the first factor being the pendulum angle
$\phi$ and the second factor being the cart position $s$. The symmetry group $G$ of the kinetic energy of the pendulum-cart system is that of translation in the $s$ variable, so $G = \mathbb{R}$. 

The length of the pendulum is $l$, the mass of the
pendulum is $m$ and that of the cart is $M$.

For the cart-pendulum system, $\alpha$, $\beta (\phi) $, $\gamma$, are given by 
\begin{equation}\label{cart_pendulum_metric.eqn} 
\alpha = m l^2, \quad \beta (\phi) = ml \cos (\phi - \psi), \quad \gamma = M + m.
\end{equation}
The potential energy is $ V (q) = V _1 (\phi) + V _2 (s) $, where 
\[
V _1 (\phi) = - mgl \cos \phi , \quad V _2 (s) = - \gamma g s \sin \psi .
\]
Note that $\alpha\gamma-\beta^2 (\phi) > 0$, and that the potential energy becomes $G$-invariant when the plane is horizontal, \emph{i.e.,} when $ \psi = 0 $.
 
Since the Lagrangian for the cart-pendulum system is of the form \eqref{continuous_11_lagr.eqn}, the discrete control laws \eqref{discrete_u.eqn} and \eqref{discrete_u_potential.eqn} stabilize the upward vertical equilibrium of the \emph{pendulum}. As in the continuous-time setting, the \emph{cart} is stabilized by symmetry-breaking controller \eqref{discrete_u_potential.eqn} and is not stabilized by symmetry-preserving controller \eqref{discrete_u.eqn}. 

Simulations of the discrete cart-pendulum system are shown in the next section.

\section{Simulations}\label{simulations.sec}

Simulating the behavior of the discrete controlled Lagrangian system
involves viewing equations \eqref{dcp1.eqn} and \eqref{controlled_lagrangian_dynamics_2.eqn} as an implict update map $\Phi:(q_{k-2},
q_{k-1})\mapsto(q_{k-1},q_k)$. This presupposes that the initial
conditions are given in the form $(q_0,q_1)$; however it is generally
preferable to specify the initial conditions as $(q_0,\dot q_0)$. 
This is achieved by solving the 
boundary condition
\[
\frac{\partial L}{\partial \dot q}(q_0,\dot q_0) + D_1 L^d(q_0,q_1) +
F^d_1(q_0,q_1)
= 0
\]
for $q_1$. Once the initial conditions are expressed in the form $(q_0,q_1)$, the discrete evolution can be obtained using the implicit update
map $\Phi$.

We first consider the case of kinetic shaping on a level surface (with $\psi=0$), when $\kappa$ is twice the critical value, and without dissipation. Here, $h=0.05\,\textrm{sec}$, $m=0.14\,\textrm{kg}$, $M=0.44\,\textrm{kg}$, and $l=0.215\,\textrm{m}$. As shown in Figure~\ref{fig:discrete_kinetic_nodiss}, the $\phi$ dynamics is stabilized, but since there is no dissipation, the oscillations are sustained. The $s$ dynamics exhibits both a drift and oscillations, as potential shaping is necessary to stabilize the translational dynamics.

When dissipation is added, the $\phi$ dynamics is asymptotically stabilized, as shown in Figure~\ref{fig:discrete_kinetic_diss}. However, even though the oscillations are damped, the $s$ dynamics retains a drift motion, as expected.

We next consider the case of potential shaping on an inclined surface (with  $\psi=\frac{\pi}{9}\,\textrm{radians}$) without dissipation, with the other physical parameters as before. Here, our goal is to regulate the cart at $s=0$ and the pendulum at $\phi=0$. 
We set $ V _\varepsilon = - \frac{\varepsilon}{2} y ^2 $. The control gains are chosen to be $\kappa=20$, $\rho=-0.02$, and $\epsilon=0.00001$. It is worth noting that the discrete dynamics remain bounded near the desired equilibrium, and this behavior persists even for significantly longer simulation runs involving $10^6$ time-steps. To more clearly visualize the dynamics, we only include a 4000 time-step segment of this computation in Figure~\ref{fig:discrete_potential_nodiss}.
The exceptional stability of the discrete controlled trajectory can presumably be understood in terms of the bounded energy oscillations characteristic of symplectic and variational integrators.

\begin{figure}[ht]
\hspace{-1.15em}
\begin{overpic}
[scale=.51]
{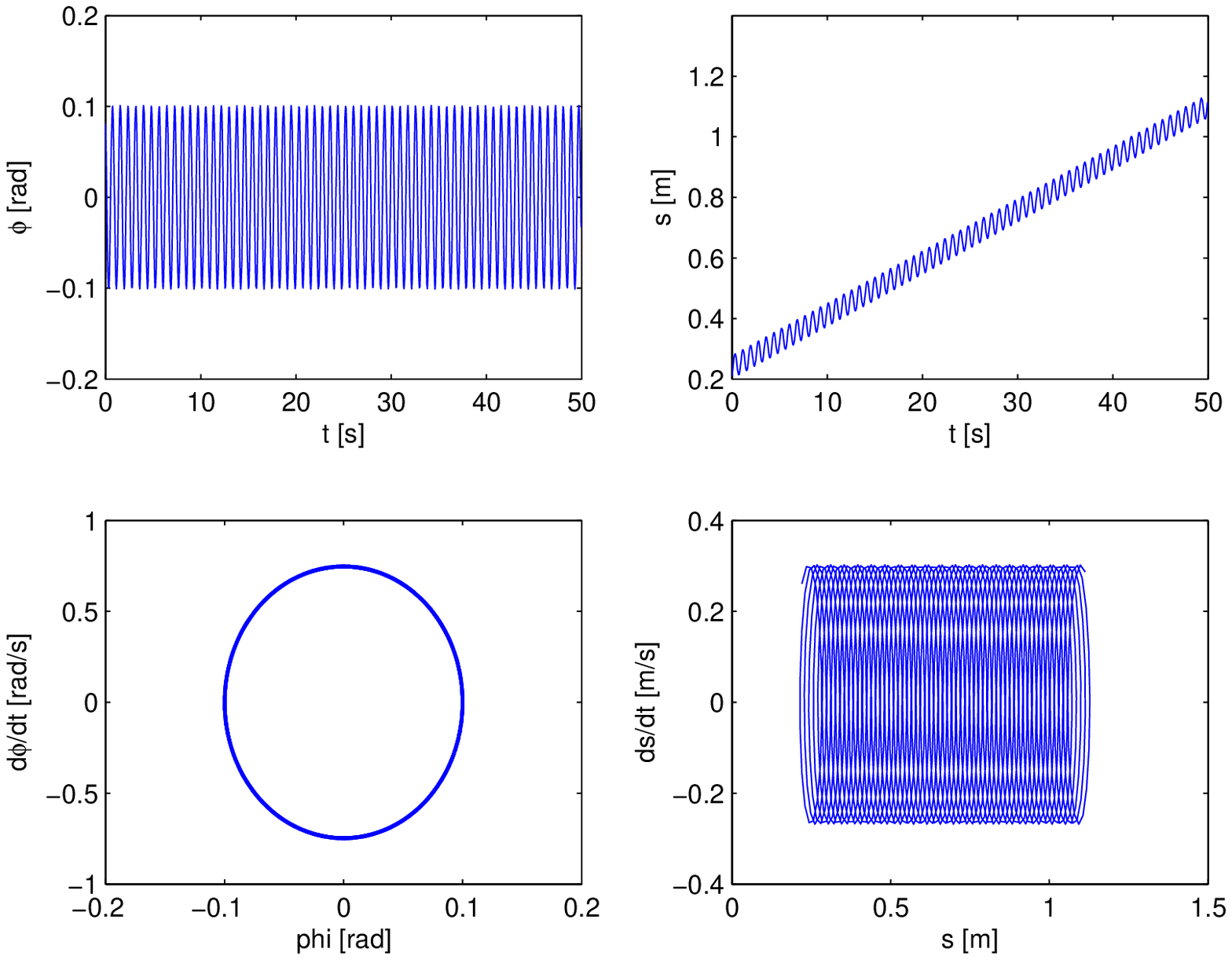}
\put(28.3,10.55)
	{\begin{rotate}{0}
	{\includegraphics[width=.015\textwidth]
	{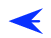}}
	\end{rotate}
	}
\put(31.5,31.67)
	{\begin{rotate}{180}
	{\includegraphics[width=.015\textwidth]
	{arrow}}
	\end{rotate}
	}
\end{overpic} 
\caption{\label{fig:discrete_kinetic_nodiss}Discrete controlled dynamics with kinetic shaping and without dissipation. The discrete controlled system stabilizes the $\phi$ motion about the equilibrium, but the $s$ dynamics is not stabilized; since there is no dissipation, the oscillations are sustained.}
\end{figure}

\begin{figure}[ht]
\hspace{-1.15em}
\begin{overpic}
[scale=.51]
{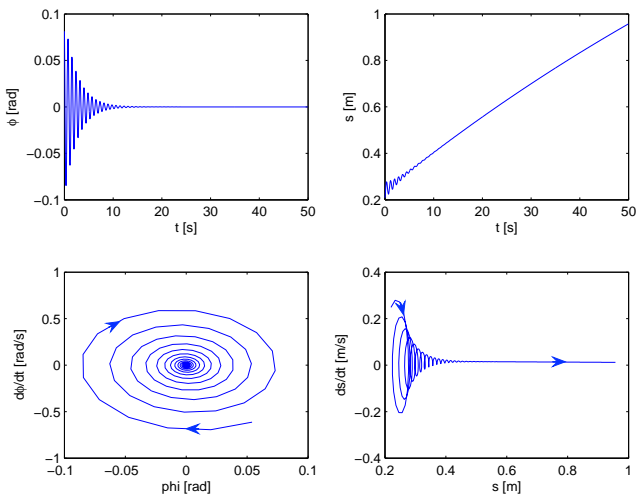}
\put(20.7,28.4)
	{\begin{rotate}{-152.3}
	{\includegraphics[width=.015\textwidth]
	{arrow}}
	\end{rotate}
	}
\put(29,11.4)
	{\begin{rotate}{-1}
	{\includegraphics[width=.015\textwidth]
	{arrow}}
	\end{rotate}
	}
\put(82.5,22.69)
	{\begin{rotate}{180}
	{\includegraphics[width=.015\textwidth]
	{arrow}}
	\end{rotate}
	}
\put(60.75,27.65)
	{\begin{rotate}{104}
	{\includegraphics[width=.015\textwidth]
	{arrow}}
	\end{rotate}
	}
\end{overpic}
\caption{\label{fig:discrete_kinetic_diss}Discrete controlled dynamics with kinetic shaping and  dissipation. The discrete controlled system asymptotically stabilizes the $\phi$ motion about the equilibrium; since there is no potential shaping, the $s$ dynamics is not stabilized, and there is a slow drift in $s$.}
\end{figure}

\begin{figure}[ht]
\hspace{-1.15em}
\begin{overpic}
[scale=.51]
{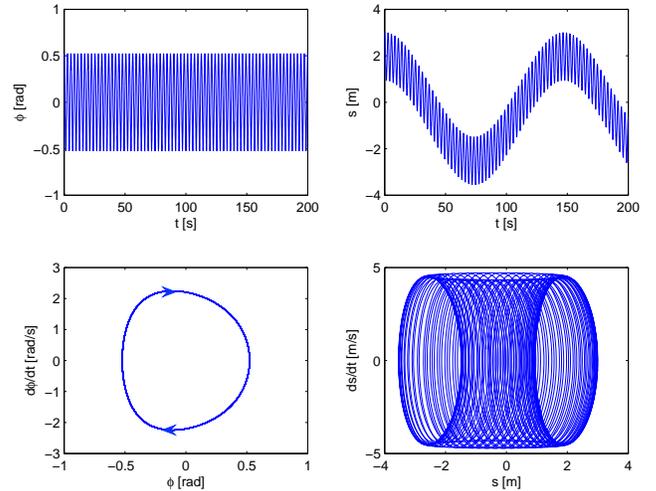}
\put(25.8,10.55)
	{\begin{rotate}{0}
	{\includegraphics[width=.015\textwidth]
	{arrow}}
	\end{rotate}
	}
\put(29,31.67)
	{\begin{rotate}{180}
	{\includegraphics[width=.015\textwidth]
	{arrow}}
	\end{rotate}
	}
\end{overpic} 
\caption{\label{fig:discrete_potential_nodiss}Discrete controlled dynamics with potential shaping and without dissipation. The discrete controlled system stabilizes the motion about the equilibrium; since there is no dissipation, the oscillations are sustained. }
\end{figure}

When dissipation is added, we obtain an asymptotically stabilizing control law, as illustrated in Figure~\ref{diss}. This is consistent with the stability analysis of Section \ref{disc_cart_pendulum.sec}.
\begin{figure}[htbp]
\hspace{-1.15em}
\begin{overpic}
[scale=.51]
{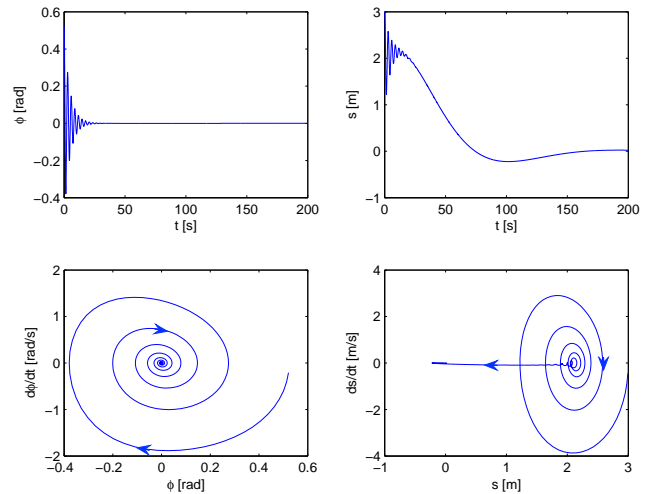}
\put(28,26.45)
	{\begin{rotate}{170}
	{\includegraphics[width=.015\textwidth]
	{arrow}}
	\end{rotate}
	}
\put(22.1,8.64)
	{\begin{rotate}{-10}
	{\includegraphics[width=.015\textwidth]
	{arrow}}
	\end{rotate}
	}
\put(70,19.85)
	{\begin{rotate}{-1}
	{\includegraphics[width=.015\textwidth]
	{arrow}}
	\end{rotate}
	}
\put(88.13,19.5)
	{\begin{rotate}{90}
	{\includegraphics[width=.015\textwidth]
	{arrow}}
	\end{rotate}
	}
\end{overpic}
\caption{Discrete controlled dynamics with potential shaping and dissipation. Here the oscillations die out and both the cart position and the pendulum angle converge to their desired values $s = 0$ and $\phi=0$.}
\label{diss}
\end{figure}

\section{Model Predictive Controller}\label{digital.sec}

We now explore the use of the forced discrete Euler--Lagrange
equations as the model in a real-time model
predictive controller, with piecewise constant control
forces. Algorithm 1 below describes the details of the procedure.
\begin{algorithm}
\caption{\textsc{Digital Controller} ( $q(\,\cdot\,), T_f, h$ )}
\begin{algorithmic}
\STATE $q_0\leftarrow$ \textbf{sense} $q(0)$
\STATE $q_1\leftarrow$ \textbf{sense} $q(h)$
\STATE $\bar q_2 \leftarrow$ \textbf{solve} $D_2 L^d(q_0, q_1) + D_1
L^d(q_1,\bar q_2) =0$ 
\STATE $\bar q_3 \leftarrow$ \textbf{solve} $D_2 L^d(q_1, \bar q_2) +
D_1 L^d(\bar q_2,\bar q_3) + F^d_1(\bar q_2, \bar q_3) =0$ 
%%%%
\STATE $u_{2+1/2} \leftarrow u\left(\frac{\bar q_2 + \bar q_3}{2},\frac{\bar q_3-\bar q_2}{h}\right)$  
\STATE \textbf{actuate} $u=u_{2+1/2}$ for $t\in[2h,3h]$
\STATE $q_2 \leftarrow$ \textbf{sense} $q(2h)$
\STATE $\bar q_3 \leftarrow$ \textbf{solve} $D_2 L^d(q_1, q_2) + D_1
L^d(q_2,\bar q_3) + F_1^d(q_2, \bar q_3) =0$ 
\STATE $\bar q_4 \leftarrow$ \textbf{solve} $D_2 L^d(q_2, \bar q_3) +
D_1 L^d(\bar q_3,\bar q_4)$
\\
\hspace*{1in}$+ F_2^d(q_2, \bar q_3) + F_1^d(\bar q_3, \bar q_4) =0$
%%%%
\STATE $u_{3+1/2} \leftarrow u\left(\frac{\bar q_3 + \bar q_4}{2},\frac{\bar q_4-\bar q_3}{h}\right)$
\STATE \textbf{actuate} $u=u_{3+1/2}$ for $t\in[3h,4h]$
\FOR{$k=4$ to $(T_f/h -1)$}
\STATE $q_{k-1} \leftarrow$ \textbf{sense} $q((k-1)h)$
\STATE $\bar q_k \leftarrow$ \textbf{solve} $D_2 L^d(q_{k-2}, q_{k-1})
+ D_1 L^d( q_{k-1},\bar q_k)$
\\
\hspace*{1in}$+ F_2^d(q_{k-2}, q_{k-1}) + F_1^d( q_{k-1}, \bar q_k) =0$
\STATE $\bar q_{k+1} \leftarrow$ \textbf{solve} $D_2 L^d(q_{k-1}, \bar
q_k) + D_1 L^d(\bar q_k,\bar q_{k+1})$
\\
\hspace*{1in}$+ F_2^d(q_{k-1}, \bar q_k) + F_1^d(\bar q_k, \bar q_{k+1}) =0$
%%%%
\STATE $u_{k+1/2} \leftarrow u\left(\frac{\bar q_k + \bar q_{k+1}}{2},\frac{\bar q_{k+1}-\bar q_k}{h}\right)$
\STATE \textbf{actuate} $u=u_{k+1/2}$ for $t\in[kh,(k+1)h]$
\ENDFOR
\end{algorithmic}
\end{algorithm}

The digital controller uses the position information it senses for
$t=-2h, -h$ to estimate the positions at $t=0,h$ during the time
interval $t=[-h,0]$. This allows it to compute a symmetric finite
difference approximation to the continuous control force $u(\phi,s,\dot\phi,\dot s)$ at $t=h/2$ using the
approximation
\begin{align*}
u_{1/2} &= u\left(\frac{\bar\phi_0+\bar\phi_1}{2},\frac{\bar s_0+\bar s_1}{2}, \frac{\bar\phi_1-\bar\phi_0}{h},\frac{\bar s_1-\bar s_0}{h}\right),
\end{align*}
where the overbar indicates that the position variable is being
estimated by the numerical model. This control is then applied as a
constant control input for the time interval $[0,h]$. This algorithm can be implemented in real-time if the two forward solves can be computed within the time interval $h$. On a 2.5\,GHz PowerPC G5 running MATLAB, two forward solves take 631.2\,$\mu$sec, which is sufficiently fast to drive a digital controller with a frequency in excess of 1.5\,kHz. In our simulation, the digital controller had a frequency of 20\,Hz, which involves a computational load that is easily accommodated by an embedded controller. 

The initialization of the discrete controller is somewhat involved,
since the system is unforced during the time interval $[0,2h]$ while
the controller senses the initial states, and computes the appropriate
control forces. Consequently, a combination of the forced and unforced discrete Euler--Lagrange equations are used to predict the initial evolution of the system.

We present the numerical simulation results for the digital controller in both the case of kinetic shaping (Figure~\ref{digital_control_kinetic}) and potential shaping (Figure~\ref{digital_control_potential}). We see that in the case of kinetic shaping, the system is \nolinebreak
asymptotically \nolinebreak stabilized \nolinebreak in only the $\phi$ variable, and the $s$ dynamics exhibits a drift, whereas in the case of potential shaping, the system is asymptotically stabilized in both the $\phi$ and $s$ variables. Notice that the use of a piecewise constant control introduces dissipation-like effects, which are reduced as the time-step is decreased.

\begin{figure}[thbp]
\hspace{-1.15em}
\begin{overpic}
[scale=.51]
{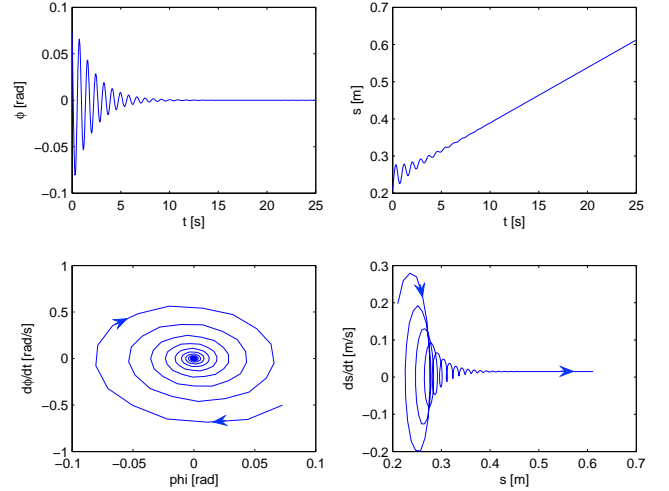} 
\put(20.7,27.9)
	{\begin{rotate}{-155.3}
	{\includegraphics[width=.015\textwidth]
	{arrow}}
	\end{rotate}
	}
\put(31.8,11.3)
	{\begin{rotate}{7}
	{\includegraphics[width=.015\textwidth]
	{arrow}}
	\end{rotate}
	}
\put(82.5,20.45)
	{\begin{rotate}{180}
	{\includegraphics[width=.015\textwidth]
	{arrow}}
	\end{rotate}
	}
\put(62.25,29.15)
	{\begin{rotate}{102}
	{\includegraphics[width=.015\textwidth]
	{arrow}}
	\end{rotate}
	}
\end{overpic}

\caption{The discrete real-time piecewise constant model predictive controller with kinetic shaping stabilizes $\phi$ to zero, but not $s$.}
\label{digital_control_kinetic}
\end{figure}

\begin{figure}[thbp]
\hspace{-1.15em}
\begin{overpic}
[scale=.51]
{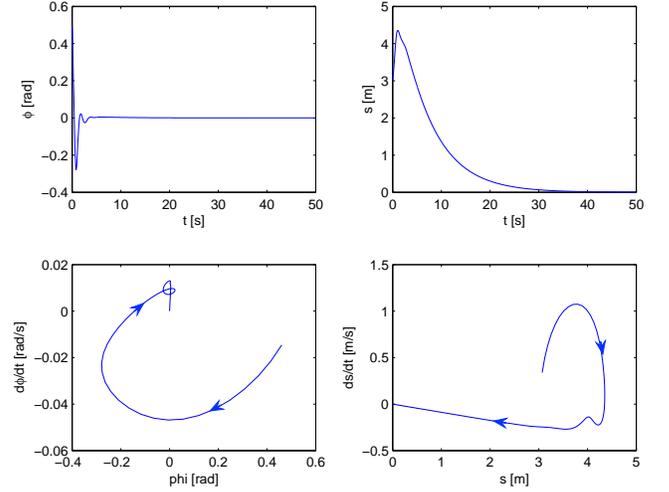} 
\put(22.7,29.9)
	{\begin{rotate}{-140.3}
	{\includegraphics[width=.015\textwidth]
	{arrow}}
	\end{rotate}
	}
\put(31.8,12.6)
	{\begin{rotate}{28.5}
	{\includegraphics[width=.015\textwidth]
	{arrow}}
	\end{rotate}
	}
\put(70,11.52)
	{\begin{rotate}{-9.3}
	{\includegraphics[width=.015\textwidth]
	{arrow}}
	\end{rotate}
	}
\put(87,21.15)
	{\begin{rotate}{100}
	{\includegraphics[width=.015\textwidth]
	{arrow}}
	\end{rotate}
	}
\end{overpic}

\caption{The discrete real-time piecewise constant model predictive controller with potential shaping stabilizes $\phi$ and $s$ to zero.}
\label{digital_control_potential}
\end{figure}

\section{Conclusions}
In this paper we have introduced potential shaping techniques  for discrete systems and have shown that these lead to an effective numerical implementation for stabilization in the case of the discrete cart-pendulum model. The method in this paper is related to other discrete methods in control that have a long history; recent papers that use discrete mechanics in the context of optimal control and celestial navigation 
are~\cite{GuBl2005}, \cite{JuMaOb2005, JuMaOb2006}, and \cite{SaShMcBl2005}. The method of discrete controlled Lagrangians for systems with higher-dimensional configuration space and with non-commutative symmetry will be developed in a forthcoming paper. 

\vspace{3.9em}

\section{Acknowledgments}
The research of AMB was supported by NSF grants DMS-0305837, DMS-0604307, and CMS-0408542. The research of ML was partially supported by NSF grant DMS-0504747 and a University of Michigan Rackham faculty grant. The research of JEM was partially supported by AFOSR Contract FA9550-05-1-0343. The research of DVZ was partially supported by NSF grants DMS-0306017 and DMS-0604108.

\end{document}